\documentclass[11pt]{amsart}

\usepackage{fullpage}
\usepackage{amsmath}
\usepackage{amssymb}
\usepackage{amsthm}
\usepackage{amscd}
\usepackage{xypic}
\usepackage{delarray}
\usepackage{longtable}
\usepackage{color}

\newtheorem{thm}{Theorem}[section]
\newtheorem{lemma}[thm]{Lemma}
\newtheorem{proposition}[thm]{Proposition}
\newtheorem{corollary}[thm]{Corollary}

\theoremstyle{definition}

\newtheorem{remark}[thm]{Remark}

\newtheorem{convention and reminder}{Convention and Reminder}
\newtheorem{convention and remark}{Convention and Remark}
\newtheorem{definition and remark}{Definition and Remark}

\newtheorem{reminders and definition}{Reminders and Definition}

\newtheorem{notation and remarks}{Notation and Remarks}
\newtheorem{notation and remark}{Notation and Remark}
\newtheorem{example}[thm]{Example}

\newcommand \h {\textup{\^}}
\newcommand \N {\ensuremath{\mathrm{\bf N}}}

\newcommand\Tor{\operatorname{Tor}}

\newcommand\codim{\operatorname{codim}}

\newcommand\reg{\operatorname{reg}}
\newcommand\Ker{\operatorname{\Ker}}

\newcommand\e{\operatorname{e}}

\newcommand\Sec{\operatorname{Sec}}

\newcommand\gin{\operatorname{gin}}

\newcommand \p {\partial}
\newcommand \pd {\operatorname{proj.dim}}
\newcommand \lra {\rightarrow}
\newcommand \llra {\longrightarrow}
\newcommand \st[1] {\stackrel{#1}{\rightarrow}}
\newcommand \cone[1] {C_{#1}(\overline{\varphi})}

\def\ds{\displaystyle}
\def\P{{\mathbb P}}
\def\Z{{\mathbb Z}}

\begin{document}

\title[Syzygies and geometric properties of projective schemes]
{On Syzygies, degree, and geometric properties of projective schemes with property $\textbf{N}_{3,p}$} 

\author[J.\ Ahn and S.\ Kwak]{Jeaman Ahn and Sijong Kwak${}^{*}$}
\address{Department of Mathematics Education, Kongju National University, 182, Shinkwan-dong, Kongju, Chungnam 314-701, Republic of Korea}
\email{jeamanahn@kongju.ac.kr}
\thanks{${}^{*}$ Corresponding author.}
\thanks{1. The first author was supported by the research grant of the Kongju National University in 2013 (No. 2013-0535)}

\address{Department of Mathematics, Korea Advanced Institute of Science and Technology,
373-1 Gusung-dong, Yusung-Gu, Daejeon, Korea}
\email{skwak@kaist.ac.kr}
\thanks{2. The second author was supported by the National Research Foundation of Korea(NRF) grant funded by the Korea government(MSIP) (No. 2013042157).}

\begin{abstract}
For an algebraic set $X$ ~(union of varieties) embedded in projective space, we say that $X$ satisfies
property $\textbf{N}_{d,p}$, $(d\ge 2)$ if the $i$-th syzygies of the homogeneous coordinate ring are generated by elements
of degree $< d+i$ for $0\le i\le p$ (see \cite{EGHP2} for details).
Much attention has been paid to linear syzygies of quadratic schemes $(d=2)$ and their geometric interpretations~(cf.~\cite{AK},\cite{EGHP1},\cite{HK},\cite{GL2},\cite{KP}).
However, not very much is actually known about the case satisfying property $\textbf{N}_{3,p}$.
In this paper, we give a sharp upper bound on the maximal length of a zero-dimensional linear section of $X$ in terms of graded
Betti numbers (Theorem~\ref{3-regulr-multisecants} (a)) when $X$ satisfies property $\textbf{N}_{3,p}$. In particular, if $p$ is the codimension $e$ of $X$ then the degree of $X$ is less than or equal to $\binom{e+2}{2}$, and
equality holds if and only if $X$ is arithmetically Cohen-Maucalay with $3$-linear resolution (Theorem~\ref{3-regulr-multisecants} (b)). This is a 
generalization of the results of Eisenbud et al. (\cite{EGHP1,EGHP2}) to the case of $\N_{3,p}$, $(p\leq e)$.
\end{abstract}


\maketitle
\tableofcontents
\thispagestyle{empty}

\section{Introduction}
Throughout this paper, we will work with a non-degenerate reduced algebraic set (union of varieties) $X$ of dimension $n$ and codimension $e$ in $\P^{n+e}$ defined over an algebraically closed field $k$ of characteristic zero.
We write  $I_X:=\bigoplus_{m=0}^{\infty} H^0(\mathcal I_X(m))$ for the defining ideal of $X$ in the polynomial ring $R=k[x_0, x_1 ,\ldots, x_{n+e}]$.

It is known that if $X$ is a variety then $X$ satisfies the condition $\deg(X)\geq e+1$ and  if equality holds then we say that $X$ has minimal degree. Del Pezzo \cite{DP} classified surfaces of minimal degree and Bertini \cite{Ber} extended the classfication to all dimensions. This classification was again extended to equidimenisonal algebraic sets that are connected in codimension one by Xamb\'{o} \cite{X}. In this case, they are all 2-regular (in the sense of Castelnuovo-Mumford) and vice versa (\cite{EG}). 

For more general algebraic sets, the inequality $\deg(X)\geq e+1$ does not hold. A simple example is a set of two skew lines in $\P^3$, which is $2$-regular.  In \cite[Theorem~0.4]{EGHP1} authors give a classification of $2$-regular  algebraic sets in terms of ``{\it smallness}''. This means that for every linear subspace $\Lambda$ such that $\Lambda\cap X$ is finite, the scheme $\Lambda\cap X$ is linearly independent. 
From this result if $X$ is $2$-regular then $\deg(X)\leq e+1$ and the extremal degree holds if and only if $X$ is a reduced aCM scheme whose defining ideal has $2$-linear resolution.

Recall that $2$-regularity means the syzygies of $I_X$ are all linear. In a different paper \cite{EGHP2} carried out by the same authors, they introduce the notion $\N_{2,p}$ (i.e. the syzygies of $I_X$ are linear for $p$ steps) and show that if $X$ satisfies the property $\N_{2,p}$ then for every linear subspace $\Lambda$ of dimension $\leq p$ such that $\Lambda\cap X$ is finite, the scheme $\Lambda\cap X$ is $2$-regular and $\deg(\Lambda\cap X)\leq \dim\Lambda+1$ (See \cite[Theorem~1.1]{EGHP2}). Since a small algebraic set is $2$-regular, this inequality implies that if $X$ satisfies the property $\N_{2,e}$ then $X$ is $2$-regular \cite[Corollary~1.8]{EGHP2}.

Summing up these results, we have the following theorem:

\begin{thm}[\cite{EGHP1, EGHP2}]\label{Eisenbud et al}
 Let $X\subset \P^{n+e}$ be a non-degenerate algebraic set (union of varieties) of dimension $n$ and codimension $e$. Suppose that $X$ satisfies property $\N_{2,p}$. 
\begin{itemize}
  \item[(a)] For a linear subspace $\Lambda$ of dimension $\leq p$, if $X\cap \Lambda$ is finite then $$\deg(X\cap \Lambda)\leq \dim\Lambda+1;$$
  \item[(b)] In particular, if $X$ satisfies property $\N_{2,e}$ (i.e. $X$ is $2$-regular) then $\deg(X)\leq e+1$ and  the extremal degree holds if and only if $X$ is a reduced aCM scheme whose defining ideal has $2$-linear resolution.
\end{itemize}

\end{thm}

In this context, it is natural to ask what we can say about algebraic sets satisfying property $\N_{3,\alpha}$ with the Green-Lazarsfeld index $p$. (i.e. $p$ is the largest $k\ge 0$ such that $X$ satisfies property $\N_{2,k}$. See \cite{BCR} for the definition).
Although linear syzygies of quadratic schemes and their geometric properties have been understood by many authors ~(cf.~\cite{AK, AR, EGHP1, HK, GL2, KP}), not very much is actually known about the geometric properties of algebraic sets satisfying $\textbf{N}_{3,\alpha}$. 

This paper is a starting point for a generalization of Theorem~\ref{Eisenbud et al} to the case of $\N_{3,\alpha}$. Our main result is the following:

\begin{thm}\label{3-regulr-multisecants}
Let $X$ be a non-degenerate algebraic set in $\P^{n+e}$ of codimension $e\ge 1$
satisfying \textit{property} $\textbf{N}_{3,\alpha}$ with the Green-Lazarsfeld index $p$. For a linear space $L^{\alpha}$ of dimension $\alpha\leq e$, 
\begin{itemize}
\item[(a)] if $X\cap L^{\alpha}$ is a finite scheme then,
\begin{equation*}
\deg(X\cap L^{\alpha})\leq 1+ \alpha+\min\left\{\frac{|\alpha-p|(\alpha+p+1)}{2}, \beta^R_{\alpha,2}(R/I_X)\right\};
\end{equation*}
\item[(b)] In particular, if $X$ satisfies property $\N_{3,e}$ then $\deg(X)\leq \binom{e+2}{2}$ and the extremal degree holds if and only if $X$ is a reduced aCM scheme whose defining ideal has $3$-linear resolution. (hence 3-regular).
\end{itemize}
\end{thm}

There is an algebraic set satisfying property $\N_{3,e}$ that is not $3$-regular (see Example~\ref{example:2014}). This means the condition $\N_{3,e}$ does not imply $3$-regularity.  Besides, we do not know a nice characterization of algebraic sets having $3$-linear resolution corresponding to the case of $2$-linear resolution (\cite[Theorem~0.4]{EGHP1}). Nevertheless, our result is just such a generalization of Theorem~\ref{Eisenbud et al} to the case of $\N_{3,\alpha}$ with an alternative perspective and approach. 

To achieve the result, we use the elimination mapping cone construction for graded modules and apply it to give a systematic approach to the relation between multisecants and graded Betti numbers. 
From the maximal bound for the length of finite linear sections of algebraic sets satisfying property $\N_{3,e}$ (in terms of the graded Betti numbers), the extremal cases can be characterized by the combinatorial property of the syzygies of generic initial ideals. 

We also provide some illuminating examples of our results and corollaries via calculations done with {\it Macaulay 2} \cite{GS}.

\section{Preliminaries}
\subsection{Notations and Definitions}
For precise statements, we begin with notations and definitions used in the subsequent sections:
\begin{itemize}
\item[$\bullet$] We work over an algebraically closed field $k$ of characteristic zero.
\item[$\bullet$] Unless otherwise stated, $X$ is a non-degenerate, reduced algebraic sets~(union of varieties) of dimension $n$ in $\P^{n+e}$.

\item[$\bullet$] For a finitely generated graded $R=k[x_0, x_1,\ldots,x_{n}]$-module $M=\bigoplus_{\nu \geq 0}M_{\nu}$, consider a minimal free resolution of $M$:
    $$\cdots \rightarrow \oplus_j R(-i-j)^{\beta^{R}_{i,j}(M)}\rightarrow\cdots\rightarrow\oplus_j R(-j)^{\beta^{R}_{0,j}(M)}\rightarrow M\rightarrow 0$$
    where $\beta^R_{i,j}(M):=\dim_k\Tor_i^{R}(M,k)_{i+j}$. We write $\beta^R_{i,j}(M)$
    as $\beta^R_{i,j}$ if it is obvious. We define the regularity of $M$ as follows:
 $$
 \reg_R(M):= \max\{j\mid \beta^R_{i,j}(M)\neq 0 \text{ for some } i\}
 $$
 In particular, $\reg(X):= \reg_R(I_X)$.

\item One says that $M$ satisfies {\it property $\N^R_{d,\alpha}$} if $\beta^R_{i,j}(M)=0$ for all $j\ge d$ and $0\le i\le\alpha$. We can also think of $M$ as a graded $S_t=k[x_t,\ldots,x_{n}]$-module by an inclusion map $S_t \hookrightarrow R$. As a graded $S_t$-module,
    we say that $M$ satisfies {\it property $\N^{S_t}_{d,\alpha}$} if $\beta^{S_t}_{i,j}(M):=\dim_k\Tor_i^{S_t}(M,k)_{i+j}=0$ for all $j\ge d$ and $0\le i\le \alpha$.

\item For an algebraic set $X$ in $\P^{n+e}$, the Green-Lazarsfeld index of $X$, denoted by $\text{index}_{\rm GL}(X)$, is the largest $\alpha\ge 0$ such that $X$ satisfies property $\N_{2,\alpha}$.
 \end{itemize}
\subsection{Elimination Mapping Cone Construction}

For a graded $R$-module $M$, consider the natural multiplicative $S_1=k[x_1,x_2,\ldots,x_{n}]$-module map $\varphi: M(-1) \stackrel{\times x_0}{\longrightarrow} M$ such that $\varphi(m)=x_0\cdot m$ and the induced map on the graded Koszul complex of $M$ over $S_1$:
$$\overline{\varphi}: \Bbb F_{\bullet}=K_{\bullet}^{S_1}(M(-1)) \stackrel{\times x_0}{\longrightarrow} \Bbb G_{\bullet}=K_{\bullet}^{S_1}(M).$$

Then, we have the mapping cone $(C_{\bullet}(\overline{\varphi}),\p_{\,\,\overline{\varphi}})$
such that $C_{\bullet}(\overline{\varphi})=\Bbb G_{\bullet}\bigoplus\Bbb F_{\bullet}[-1],$ and $W=\langle x_1, x_2,\ldots,x_{n} \rangle;$
\begin{itemize}
\item $C_i(\overline{\varphi})_{i+j}=[\Bbb G_{i}]_{i+j}\bigoplus
[\Bbb F_{i-1}]_{i+j}=\left(\wedge^i W\otimes M_j\right)\oplus \left(\wedge^{i-1} W\otimes M_j\right)$.\\
\item the differential $\p_{\,\,\overline{\varphi}}: C_i(\,\,\overline{\varphi})
\rightarrow C_{i-1}(\overline{\varphi})$ is given by
\[\p_{\,\,\overline{\varphi}}=\left(%
\begin{array}{lr}
  \p & \overline{\varphi} \\
  0 & -\p \\
\end{array}%
\right),\] where $\p$ is the differential of Koszul complex $K_{\bullet}^{S_1}(M)$.

\end{itemize}

From the exact sequence of complexes
\begin{equation}
0\longrightarrow \Bbb G_{\bullet} \longrightarrow C_{\bullet}(\overline{\varphi}) \longrightarrow \Bbb F_{\bullet}[-1]\longrightarrow 0
\end{equation}
and the natural isomorphism $H_i(\cone{\bullet})_{i+j}\simeq \Tor_i^R(M,k)_{i+j}$ (cf. Lemma~3.1 in \cite{AK}),
we have the following long exact sequence in homology.
\begin{thm}[Theorem~3.2 in \cite{AK}]\label{Thm:Mapping_cone_theorem}  For a graded $R$-module $M$, there is a long exact sequence:
\begin{equation*}
\begin{array}{ccccccccccccc}
\llra\Tor_{i}^{S_1}(M,k)_{i+j}\llra\Tor_{i}^R(M,k)_{i+j}\llra
\Tor_{i-1}^{S_1}(M,k)_{i-1+j}\llra&  &&
\\[2ex]
\st{\delta=\times x_0}\Tor_{i-1}^{S_1}(M,k)_{i-1+j+1}\llra\Tor_{i-1}^R(M,k)_{i-1+j+1}\llra
\Tor_{i-2}^{S_1}(M,k)_{i-2+j+1}&
\end{array}
\end{equation*}
whose connecting homomorphism $\delta$ is the multiplicative map $\times\, x_0$.
\end{thm}

\begin{corollary}\label{projective dimension}
Let $M$ be a finitely generated graded $R$-module and also finitely generated as an $S_1$-module.
Then, $$\pd_{S_1}(M)=\pd_R(M)-1.$$
\end{corollary}
\begin{proof}
Let $\ell=\pd_R(M)$. Thus,  $\beta^{R}_{\ell+1,j}(M)=0$ for all $j\ge 1$ and the following map $\delta=\times x_0$ is injective for all $j\ge 1$:
\[0=\Tor_{\ell +1}^{R}(M,k)_{\ell+1+j} \lra \Tor_{\ell}^{S_1}(M,k)_{\ell+j} \st{\delta=\times x_0} \Tor_{\ell}^{S_1}(M,k)_{\ell+j+1}.\]
But, $\Tor_{\ell}^{S_1}(M,k)_{\ell+j+1}=0$ for $j>>0$ due to finiteness of $M$ (as an $S_1$-module). Therefore,  $\Tor_{\ell}^{S_1}(M,k)_{\ell+j}=0$ for all $j\ge 1$.
On the other hand, $\beta^{R}_{\ell,j_{*}}(M)\neq 0$ for some $j_{*}>0$. So,
\[0=\Tor_{\ell}^{S_1}(M,k)_{\ell+j_{*}} \lra \Tor_{\ell}^{R}(M,k)_{\ell+j_{*}}\st{} \Tor_{\ell-1}^{S_1}(M,k)_{\ell-1+j_{*}}\]
is injective and $\beta^{S_1}_{\ell-1,j_{*}}(M)\neq 0$. Consequently, we get
$$\pd_{S_1}(M)=\pd_R(M)-1,$$
as we wished.
\end{proof}

\begin{proposition}\label{prop: N_dp_as_S_module}
Let $M$ be a finitely generated graded $R$-module satisfying property $\N^R_{d,\alpha}, (\alpha\geq 1)$.
If $M$ is also finitely generated as an $S_1$-module, then we have the following:
\begin{itemize}
\item[(a)] $M$ satisfies property $\N^{S_1}_{d,\alpha-1}$. In particular,  $\reg_{S_1}(M)=\reg_{R}(M)$.
\item[(b)] $\beta^{S_1}_{i-1,d-1}(M)\leq \beta^{R}_{i,d-1}(M)$ for $1\le i\le \alpha$.
\end{itemize}
\end{proposition}
\begin{proof}
Suppose that $M$ satisfies  $\N^R_{d,\alpha}, (\alpha\geq 1)$ and let $1\le i\le \alpha$ and $j\geq d$.\\
(a): Consider the exact sequence from Theorem ~\ref{Thm:Mapping_cone_theorem}
\begin{equation}
\begin{array}{ccccccccccccc}
\cdots \lra \Tor_{i}^R(M,k)_{i+j}& \lra
\Tor_{i-1}^{S_1}(M,k)_{i-1+j}&\st{\delta=\times x_0}&&
\\[2ex]
&\Tor_{i-1}^{S_1}(M,k)_{i-1+j+1}&\lra& \Tor_{i-1}^R(M,k)_{i-1+j+1}&\lra & \cdots
\end{array}
\end{equation}
By the property $N^R_{d,\alpha}$, we see that
$\Tor_{i}^R(M,k)_{i+j}=0$. Hence we obtain an isomorphism
\[\Tor_{i-1}^{S_1}(M,k)_{(i-1)+j}\st{\delta=\times x_0}\Tor_{i-1}^{S_1}(M,k)_
{i+j}.\]
By assumption that $M$ is a finitely generated $S_1$-module, we conclude that $\Tor_{i-1}^{S_1}(M,k)_{(i-1)+j}=0$
for $1\le i\le \alpha $ and $j\geq d$. Hence $M$ satisfies $\N^{S_1}_{d,\alpha-1}$. If $\alpha=\infty$, we have that $\reg_{S_1}(M)=\reg_{R}(M)$.\\
(b): Note that we have the following surjection map for $1\le i\le \alpha$
\[\Tor_{i}^{R}(M,k)_{i+d-1} \lra \Tor_{i-1}^{S_1}(M,k)_{i-1+d-1} \st{\delta=\times x_0} \Tor_{i-1}^{S_1}(M,k)_{i-1+d}=0,\]
which is obtained from Theorem~\ref{Thm:Mapping_cone_theorem}. This implies that for $1\le i\le \alpha$
\begin{align*}
\beta^{S_1}_{i-1,d-1}(M)\leq \beta^R_{i,d-1}(M)
\end{align*}
as we wished.
\end{proof}

\begin{remark}\label{Remk:1}
Let $M$ be a finitely generated graded $R$-module satisfying property $\textbf {N}^{R}_{d,\alpha}$ for some $\alpha\geq 1$. If $M$ is also
finitely generated as an $S_t=k[x_{t}, x_{t+1}\ldots,x_n]$-module for every $1\leq t\leq \alpha$ then $M$ satisfies property $\textbf {N}^{S_t}_{d,\alpha-t}$. Moreover, in the strand of $j=d-1$, we have the inequality
\[\beta^{S_{\alpha}}_{0,d-1}~\le \beta^{S_{\alpha-1}}_{1,d-1}~\le \cdots\le \beta^{S_1}_{\alpha-1,d-1}~\le \beta^R_{\alpha,d-1},\]
which follows from Proposition ~\ref{prop: N_dp_as_S_module}~(b).
\end{remark}

The most interesting case is a projective coordinate ring $M=R/I_X$ of an algebraic set $X$. In this case, the elimination mapping cone theorem is
naturally associated to outer projections of $X$. Our starting point is to understand some algebraic and geometric information on $X$ via the relation between $\Tor_i^R(R/I_X,k)$ and $\Tor_i^{S_{\alpha}}(R/I_X,k)$.
\bigskip

Let $X$ is a non-degenerate algebraic set of dimension $n$ in $\P^{n+e}$. Let $\Lambda=\P^{\alpha-1}$ be an $(\alpha-1)$-dimensional linear subspace with homogeneous coordinates $x_0, \ldots, x_{\alpha-1}$, $(\alpha\le e)$ such that $\Lambda\cap X$ is empty.
Then each point $q_i=[0:0:\cdots:1:\cdots:0]$ whose $i$-th coordinates is $1$ is not contained in $X$ for $0\le i \le \alpha-1$.
Therefore, there is a homogeneous polynomial $f_i \in I_X$ of the form $x_i^{m_i}+ g_i$ where $g_i\in R=k[x_0,x_1,\ldots, x_{n+e}]$
is a homogeneous polynomial of degree $m_i$ with the power of $x_i$ less than $m_i$. Therefore, $R/I_X$ is a finitely generated $S_{\alpha}=k[x_{\alpha}, x_{\alpha+1},\ldots, x_{n+e}]$-module with monomial generators
$$\{x_0^{j_0}x_1^{j_1}\ldots x_{\alpha-1}^{j_{\alpha-1}}\mid 0\le j_k < m_k, 0\le k\le {\alpha-1}\}.$$
Note that the above generating set is not minimal. If $X$ satisfies $\textbf {N}^{R}_{d,\alpha}$ then $X$ also satisfies $\textup{\textbf{N}}^{S_{\alpha}}_{d, 0}$. This implies that $R/I_X$ is generated in degree $< d$ as a $S_{\alpha}$-module and thus $\beta^{S_{\alpha}}_{0,i}\le \binom{\alpha-1+i}{i}$ for $0\le i\le d-1$. To sum up, we have the following corollary.

\begin{corollary}\label{N_{d,p-1}as a S-module}
Suppose $X$ satisfies the property ${\textup{\textbf {N}}}^R_{d,\alpha}$ and consider the following minimal free resolution of $R/I_X$ as a graded $S_{\alpha}=k[x_{\alpha},\ldots, x_{n+e}]$-module:
$$\cdots\rightarrow F_{1} \rightarrow F_0\rightarrow R/I_X \rightarrow 0.$$
\begin{itemize}
 \item[(a)] $R/I_X$ satisfies the property $\textup{\textbf{N}}^{S_{\alpha}}_{d, 0}$ as an $S_{\alpha}$-module;
 \item[(b)] The Betti numbers of $F_0$ satisfy the following:
\begin{itemize}
\item[(i)] $\beta^{S_{\alpha}}_{0,0}=1$, $\beta^{S_{\alpha}}_{0,1}=\alpha$, and $\beta^{S_{\alpha}}_{0,i}\le \binom{\alpha-1+i}{i}$ for $2\le i\le d-1$;
\item[(ii)] Furthermore, $\beta^{S_{\alpha}}_{0,d-1}~\le \beta^{S_{\alpha-1}}_{1,d-1}~\le \cdots\le \beta^{S_1}_{\alpha-1,d-1}~\le \beta^R_{\alpha,d-1}$.
\end{itemize}
\item[(c)] When $\alpha=e$, $R/I_X$ is a free $S_{e}$-module if and only if $X$ is arithmetically Cohen-Macaulay.
In this case, letting $d=\reg(X)$, $$R/I_X=S_{e}\oplus S_{e}(-1)^e \oplus \cdots\oplus S_e(-d+1)^{\beta^{S_e}_{0,d-1}}$$
and ${\pi_{\Lambda}}_{*}\mathcal O_X =\mathcal O_{\P^n}\oplus\mathcal O_{\P^n}(-1)^{e}\oplus\cdots\oplus\mathcal O_{\P^n}(-d+1)^{\beta^{S_e}_{0,d-1}}$.
\end{itemize}
\end{corollary}
\begin{proof}
Note that $\binom{\alpha-1+i}{i}$ is the dimension of the vector space of all homogeneous polynomials of degree $i$ in
$k[x_{0},\ldots, x_{\alpha-1}]$ defining $\Lambda=\P^{\alpha-1}$. Since $X$ is non-degenerate, $\{x_i\mid 0\le i\le \alpha-1\}$
is contained in the minimal generating set of $R/I_X$ as an $S_{\alpha}$-module. So, $\beta^{S_{\alpha}}_{0,1}=\alpha$.
The remaining part of $(b)$ is given by Proposition~\ref{prop: N_dp_as_S_module} and the argument given in Remark ~\ref{Remk:1} below.

For a proof of $(c)$,
first note that by Corollary ~\ref{projective dimension} and Corollary ~\ref{prop: N_dp_as_S_module},
\[
\begin{array}{rcl}
 \pd_{S_e}(R/I_X)&=&\pd_R(R/I_X)-e\\
\reg_{S_e}(R/I_X)&=&\reg_{R}(R/I_X).
 \end{array}
\]
Consequently, $R/I_X$ is a free $S_{e}$-module if and only if $\pd_R(R/I_X)=e$, as we wished.
\end{proof}

\begin{remark}\label{Remk:2}
If a reduced algebraic set $X$ is arithmetically Cohen-Macaulay, then it is locally Cohen-Macaulay, equidimensional and connected in codimension one.
Furthermore, as shown in the Corollary ~\ref {N_{d,p-1}as a S-module},
$${\pi_{\Lambda}}_{*}\mathcal O_X =\mathcal O_{\P^n}\oplus\mathcal O_{\P^n}(-1)^{e}\oplus\cdots\oplus\mathcal O_{\P^n}(-d+1)^{\beta^{S_e}_{0,d-1}}.$$

However, in general, if $X$ is locally Cohen-Macaulay and equidimensional, then ${\pi_{\Lambda}}_{*}\mathcal O_X$ is a vector
bundle of rank $r=\deg(X)$. Furthermore, by the well-known splitting criterion due to Horrock or Evans-Griffith(\cite {EvG}, \cite{H}),
${\pi_{\Lambda}}_{*}\mathcal O_X$ is a direct sum of line bundles if and only if $H^i(\P^{n},{\pi_{\Lambda}}_{*}\mathcal O_X(j))=H^i(X, \mathcal O_{X}(j))=0$
for all $1\le i \le n-1, \forall j\in \Z$. This condition is weaker than arithmetically Cohen-Macaulayness.
\end{remark}

\begin{example}[Macaulay 2 \cite{GS}]\label{example:1}
For one's familiarity with these topics, we show the simplest examples in the following table: Let $\Lambda=\P^{i-1}$ be a general linear subspace
with coordinates $x_0, \cdots, x_{i-1}$ and $R/I$ is a $S_i=k[x_i, \cdots, x_{n+e}]$-module. Note that by Corollary~\ref{projective dimension} and Proposition~\ref{prop: N_dp_as_S_module},
$$\pd_{S_i}(R/I_X)=\pd_R(R/I_X)-i \text { and } \reg_{S_i}(R/I_X)=\reg_{R}(R/I_X).$$

{\tiny
\begin{longtable}{l|| c  c  c  c l}
\hline\\
 & {\bf $R$-modules} & {\bf $S_1$-modules} & {\bf $S_2$-modules} \\
\hline\\
 {\tiny\texttt{
 \begin{tabular}{llllll}
 A rational noraml curve $C\subset \P^4$\\
 in generic coordinates \\
 \\
 \end{tabular}
}} & {\tiny\texttt{
 \begin{tabular}{c|ccccccccccccccccccccccccccccc}
           &0 & 1 & 2 & 3 \\[1ex]
        \hline\\
        0 & 1 & 0 &  0 & 0  \\[1ex]
        1 & 0 & 6 &  8 & 3  \\[1ex]
  \end{tabular}
}}  & {\tiny\texttt{
 \begin{tabular}{l|lllll}
           &0 & 1 & 2 &\\[1ex]
        \hline\\
        0 & 1 & 0 & 0\\[1ex]
        1 & 1 & 5 & 3 \\[1ex]
  \end{tabular}
}} & {\tiny\texttt{
 \begin{tabular}{l|lllll}
           &0 & 1 \\[1ex]
        \hline\\
        0 & 1 & 0 \\[1ex]
        1 & 2 & 3  \\[1ex]
  \end{tabular}
}}\\[1ex]

\hline\\
{\tiny\texttt{
 \begin{tabular}{llllll}
A generic complete intersection\\
$S\subset \P^4$ of quadric and cubic\\
 \\
 \end{tabular}
}}
& {\tiny\texttt{
 \begin{tabular}{c|ccccccccccccccccccccccccccccc}
           &0 & 1 & 2\\[1ex]
        \hline\\
        0 & 1 & 0 & 0 \\[1ex]
        1 & 0 & 1 & 0 \\[1ex]
        2 & 0 & 1 & 0 \\[1ex]
        3 & 0 & 0 & 1 \\[1ex]  \end{tabular}
}} & {\tiny\texttt{
 \begin{tabular}{c|ccccccccccccccccccccccccccccc}
           &0 & 1 \\[1ex]
        \hline\\
        0 & 1 & 0 \\[1ex]
        1 & 1 & 0 \\[1ex]
        2 & 0 & 1 \\[1ex]
        3 & 0 & 1 \\[1ex]
  \end{tabular}
}} & {\tiny\texttt{
 \begin{tabular}{c|ccccccccccccccccccccccccccccc}
           &0 \\[1ex]
        \hline\\
        0 & 1 \\[1ex]
        1 & 2 \\[1ex]
        2 & 2 \\[1ex]
        3 & 1 \\[1ex]
  \end{tabular}
}} & \\[1ex]
\hline\\

{\tiny\texttt{
 \begin{tabular}{llllll}
The secant variety of a rational\\
 normal curve ${\rm Sec}(C)\subset \P^5$ in generic\\ coordinates
 \\
 \end{tabular}
}}
 & {\tiny\texttt{
  \begin{tabular}{c|ccccccccccccccccccccccccccccc}
            &0 & 1 & 2 &\\[1ex]
        \hline\\
        0 & 1 & 0 & 0 \\[1ex]
        1 & 0 & 0 & 0 \\[1ex]
        2 & 0 & 4 & 3 \\[1ex]
  \end{tabular}
}}
&
{\tiny\texttt{
 \begin{tabular}{c|ccccccccccccccccccccccccccccc}
            &0 & 1\\[1ex]
        \hline\\
        0 & 1 & 0 \\[1ex]
        1 & 1 & 0 \\[1ex]
        2 & 1 & 3  \\[1ex]
  \end{tabular}
}}
&
{\tiny\texttt{
 \begin{tabular}{c|ccccccccccccccccccccccccccccc}
            &0 \\[1ex]
        \hline\\
        0 & 1 \\[1ex]
        1 & 2 \\[1ex]
        2 & 3  \\[1ex]
  \end{tabular}
}} & \\[1ex]
\hline
\end{longtable}
}

\end{example}
In generic coordinates, Betti table for $R/I$  as a $S_i$-module can be computed with Macaulay 2 (\cite{GS}) as follows:\\

{\small\texttt{
\begin{tabular}{lllllllllll}
 minresS& = (I,i) -> (\\
            &R := ring I; \\
            &n := \# gens R;\\
            &RtoR := map(R,R,random(R\h \{0\}, R\h \{numgens R:-1\}));\\
            &S := (coefficientRing R)[apply(n-i, j -> (gens R)\#(j+i))];\\
            &F := map(R,S);\\
            &use R;\\
            &betti res pushForward(F, coker gens RtoR I)\\
            & );
\end{tabular}
}}

\section{Syzygetic properties of algebraic sets satisfying property $\textbf{N}_{3,e}$}
For an algebraic set $X$ of dimension $n$ in $\P^{n+e}$ satisfying property $\textbf {N}_{2,p}$, it is proved by Eisenbud et al in \cite{EGHP2} that
a finite scheme $X\cap \Lambda$ for any linear space $\Lambda$ of dimension $\le p$, is in general linear position of length at most $\dim\Lambda + 1$.
In addition, they show the syzygetic rigidity, i.e. $X$ satisfies property $\textbf {N}_{2, e}$ if and only if $X$ is $2$-regular.

In this section, we give a proof of Theorem~\ref{3-regulr-multisecants}. This result give us a sharp upper bound on the maximal length of a zero-dimensional linear section of $X$ in terms of graded Betti numbers when $X$ satisfies property $\textbf{N}_{3,p}$. In particular, if $p$ is the codimension $e$ of $X$ then $\deg(X)$ is at most $\binom{e+2}{2}$ and the equality holds if and only if $X$ is an arithmetically Cohen-Maucalay scheme with $3$-linear resolution.

\subsection{The proof of Theorem~\ref{3-regulr-multisecants} (a)}

Let $X$ be a non-degenerate algebraic set in $\P^{n+e}$ of codimension $e\ge 1$
satisfying \textit{property} $\textbf{N}_{3,\alpha}$ with the Green-Lazarsfeld index $p$. For a linear space $L^{\alpha}$ of dimension $\alpha\leq e$, we have to show that if $X\cap L^{\alpha}$ is a finite scheme then
\begin{equation}\label{main_formula}
\deg(X\cap L^{\alpha})\leq 1+ \alpha+\min\left\{\frac{|\alpha-p|(\alpha+p+1)}{2}, \beta^R_{\alpha,2}(R/I_X)\right\}.
\end{equation}

\begin{proof}[The proof of Theorem~\ref{3-regulr-multisecants} (a)] Note that $\beta^R_{\alpha, 2}=0$ if $\alpha\le p$. In this case, the inequality $\eqref{main_formula}$ follows from \cite[Theorem 1.1]{EGHP2} directly.
Now we assume $\alpha > p$ and $\beta^R_{\alpha, 2}\neq 0$.
Suppose $\dim(X\cap L^{\alpha})=0$ and choose a linear subspace $\Lambda\subset L^{\alpha}$ of dimension $(\alpha-1)$
disjoint from $X$ with homogeneous coordinates $x_0,\ldots, x_{\alpha-1}$.

Our main idea is to consider a projection $\pi_\Lambda:X\to \pi_\Lambda(X)\subset \P^{n+e-\alpha}$ and to regard $L^{\alpha}\cap X$
as a fiber of $\pi_\Lambda$ at the point $\pi_\Lambda(L^{\alpha}\setminus \Lambda)\in \pi_\Lambda(X)$.
From Corollary~\ref{N_{d,p-1}as a S-module} (a), we see that $R/I_X$ is finitely generated as an $S_{\alpha}=k[x_{\alpha},x_{\alpha+1}\ldots,x_{n+e}]$-module satisfying property $\N^{S_{\alpha}}_{3,0}$. Thus, the minimal free resolution of $R/I_X$ is of the following form:
\begin{equation}\label{eq:mfr}
\cdots \longrightarrow S_{\alpha}\oplus S_{\alpha}(-1)^{\alpha}\oplus S_{\alpha}(-2)^{\beta^{S_{\alpha}}_{0,2}}\longrightarrow R/I_X\longrightarrow 0.
\end{equation}
Sheafifying the sequence (\ref{eq:mfr}), we have the following surjective morphism
$$
\cdots\rightarrow \mathcal O_{\P^{n+e-\alpha}}\oplus\mathcal O_{\P^{n+e-\alpha}}(-1)^{\alpha}\oplus
\mathcal O_{\P^{n+e-\alpha}}(-2)^{\beta^{S_{\alpha}}_{0,2}}\st{\widetilde{\varphi_{\alpha}}}\pi_{{\Lambda}_{*}}{\mathcal O_X}\longrightarrow 0.
$$
For any point $q\in \pi_{\Lambda}(X)$, note that $\pi_{{\Lambda}_{*}}{\mathcal O_X}\otimes k(q)\simeq
H^0(\langle \Lambda, q\rangle, {\mathcal O}_{{\pi_{\Lambda}}^{-1}(q)})$.
Thus, by tensoring $\mathcal O_{\P^{n+e-\alpha}}(2)\otimes k(q)$ on both sides,
we have the surjection on vector spaces:
\begin{equation}\label{eq:surjection1}
[\mathcal O_{\P^{n+e-\alpha}}(2)\oplus\mathcal O_{\P^{n+e-\alpha}}(1)^{\alpha}\oplus\mathcal O_{\P^{n+e-\alpha}}^{\beta^{S_{\alpha}}_{0,2}}]\otimes k(q)\twoheadrightarrow H^0(\langle \Lambda, q\rangle, {\mathcal O}_{{\pi_{\Lambda}}^{-1}(q)}(2)).
\end{equation}
Therefore, $\langle \Lambda, q\rangle\cap X$ is $3$-regular and the length of any fiber of $\pi_{\Lambda}$ is at most
$1+\alpha+\beta^{S_{\alpha}}_{0,2}$. Hence it is important to get an upper bound of $\beta^{S_{\alpha}}_{0,2}$.\\

{\bf{Claim.}}\label{Claim:Betti numbers inequality} There are following inequalities on graded Betti numbers:
\begin{enumerate}
\item[(i)] $\beta^{S_{\alpha}}_{0,2}~\le \beta^{S_{\alpha-1}}_{1,2}~\le \cdots\le \beta^{S_1}_{\alpha-1,2}~\le \beta^R_{\alpha,2},
~~~~~~ \alpha\le e=\text{codim}(X)$~~;
\item[(ii)] $\beta^{S_{\alpha}}_{0,2}~\le \frac{(\alpha-p)(\alpha+p+1)}{2}$.
\end{enumerate}
Due to the claim, we have the following inequality:
\[\beta^{S_{\alpha}}_{0,2}~\leq \min\{\frac{|\alpha-p|(\alpha+p+1)}{2}, \beta^R_{\alpha,2}(R/I_X)\}.\]
Therefore, the length of any fiber of $\pi_{\Lambda}: X \to \mathbb P^{n+e-\alpha}$ is at most
$$1+\alpha+\beta^{S_{\alpha}}_{0,2}\leq 1+\alpha+\min\{\frac{|\alpha-p|(\alpha+p+1)}{2}, \beta^R_{\alpha,2}(R/I_X)\},$$
which completes a proof of Theorem ~\ref{3-regulr-multisecants}.\\

Now let us prove the Claim. Note that Claim $({\rm i})$ follows directly from Corollary~\ref{N_{d,p-1}as a S-module} (b) for $d=3$.
Hence we only need to show Claim $({\rm ii})$.
We consider the multiplicative maps appearing in the mapping cone sequence as follows:
\begin{equation}\label{eq:long exact sequence of tor}
\begin{array}{cccccccccccccc}
\Tor_{0}^{S_{\alpha}}(R/I_X,k)_{1}\st{\times x_{\alpha-1}}\Tor_{0}^{S_{\alpha}}(R/I_X,k)_{2}\twoheadrightarrow\Tor_{0}^{S_{\alpha-1}}(R/I_X,k)_{2}\lra 0,\\[1ex]
\Tor_{0}^{S_{\alpha-1}}(R/I_X,k)_{1}\st{\times x_{\alpha-2}}\Tor_{0}^{S_{\alpha-1}}(R/I_X,k)_{2}\twoheadrightarrow\Tor_{0}^{S_{\alpha-2}}(R/I_X,k)_{2}\lra 0,\\[1ex]
\cdots~~~\cdots~~~\cdots\\[1ex]
\Tor_{0}^{S_{p+1}}(R/I_X,k)_{1}\st{\times x_{p}}\Tor_{0}^{S_{p+1}}(R/I_X,k)_{2}\twoheadrightarrow\Tor_{0}^{S_{p}}(R/I_X,k)_{2}=0.
\end{array}
\end{equation}
Since $R/I_X$ satisfies property $\textbf {N}^{S_p}_{2,0}$ as an $S_p$-module, we get $\Tor_{0}^{S_{p}}(R/I_X,k)_{2}=0$.
From the above exact sequences, we have the following inequalities on the graded Betti numbers by dimension counting:

\[\beta^{S_{\alpha}}_{0,2}\le \beta^{S_{\alpha}}_{0,1}+\beta^{S_{\alpha-1}}_{0,2}\le \beta^{S_{\alpha}}_{0,1}+\beta^{S_{\alpha-1}}_{0,1}+\beta^{S_{\alpha-2}}_{0,2}\le\cdots\le \beta^{S_{\alpha}}_{0,1}+\beta^{S_{\alpha-1}}_{0,1}+\cdots\beta^{S_{p+1}}_{0,1}= \]
\[\alpha +(\alpha-1)+\cdots + (p+1)=\frac{({\alpha-p})(\alpha+p+1)}{2}. \]
Thus, we obtain the desired inequality
$$\beta^{S_{\alpha}}_{0,2}(R/I_X)\le \min\{\frac{{(\alpha-p)}(\alpha+p+1)}{2}, ~~~\beta^R_{\alpha,2}(R/I_X)\},$$
as we claimed.
\end{proof}

\begin{remark}\label{remk:2}
In the proof of Theorem~\ref{3-regulr-multisecants} (a), we consider the following surjection:
$$[\mathcal O_{\P^{n-\alpha}}(2)\oplus\mathcal O_{\P^{n-\alpha}}(1)^{\alpha}\oplus\mathcal O_{\P^{n-\alpha}}^{\beta^{S_{\alpha}}_{0,2}}]\otimes k(q)
\twoheadrightarrow \,\,\,H^0(\langle \Lambda, q\rangle, {\mathcal O}_{{\pi_{\Lambda}}^{-1}(q)}(2))$$
where $[\mathcal O_{\P^{n-\alpha}}(2)\oplus\mathcal O_{\P^{n-\alpha}}(1)^{\alpha}\oplus\mathcal O_{\P^{n-\alpha}}^{\beta^{S_{\alpha}}_{0,2}}]\otimes k(q)
\subset H^0(\langle \Lambda, q\rangle,\mathcal O_{\langle \Lambda, q\rangle}(2))$.
Thus, $\pi_{\Lambda}^{-1}(q)=X\cap \langle \Lambda, q\rangle$ is $2$-normal and so $3$-regular.
Similarly, we can show that if $X$ satisfies $\N_{d,\alpha}$, then every finite linear section $X\cap L^{\alpha}$ is $d$-regular, which was first proved
by Eisenbud et al\cite[Theorem 1.1]{EGHP2}.
Moreover, from the following surjection as an $S_{\alpha}$-module
$$S_{\alpha}\oplus S_{\alpha}(-1)^{\alpha}\oplus S_{\alpha}(-2)^{\beta^{S_{\alpha}}_{0,2}}\oplus S_{\alpha}(-3)^{\beta^{S_{\alpha}}_{0,3}}\cdots\oplus S_{\alpha}(-d+1)^{\beta^{S_{\alpha}}_{0,d-1}}\rightarrow R/I_X\rightarrow 0,$$
we see that if $X$ satisfies property $\N_{d,\alpha}$ then $\deg(X\cap L^{\alpha})\leq 1+\alpha+ \ds\sum_{t=2}^{d-1} \beta^{S_{\alpha}}_{0,t}.$
\end{remark}

The following result shows that if $X$ is a nondegenerate variety satisfying $\N_{3,\e}$ then there is some sort of rigidity toward the beginning and the end of the resolution. This means the following Betti diagrams are equivalent; 
\begin{center}
  \begin{tabular}{ccccccc}
Property $\N_{3,e}$ and $\beta^R_{e,2}=0$ &  & $X$ is $2$-regular\\[2ex]
{\tiny\texttt{
 \begin{tabular}{c|ccccccccccccccccccccccccccccc}
        0 & 1 & 2 & ... & e-1 & e & e+1 & e+2 & ... \\[1ex]
        \hline\\
        0 & 1 & 0 & ... & 0   & 0 & 0   & 0 & ...\\[1ex]
        1 & 0 & * & ... & *   & * & *   & * & ...\\[1ex]
        2 & 0 & * & ... & *   & 0 & *   & * & ...\\[1ex]
        3 & 0 & 0 & ... & 0   & 0 & *   & * & ...\\[1ex]
        4 & 0 & 0 & ... & 0   & 0 & *   & * & ...\\[1ex]
 \end{tabular}
}}

& $\Longleftrightarrow$
& {\tiny\texttt{
 \begin{tabular}{c|ccccccccccccccccccccccccccccc}
        0 & 1 & 2 & ... & e-1 & e & e+1 & e+2& ... \\[1ex]
        \hline\\
        0 & 1 & 0 & ... & 0   & 0 & 0   & 0  & ... \\[1ex]
        1 & 0 & * & ... & *   & * & *   & *  & ... \\[1ex]
        2 & 0 & 0 & ... & 0   & 0 & 0   & 0  & ... \\[1ex]
        3 & 0 & 0 & ... & 0   & 0 & 0   & 0  & ... \\[1ex]
        4 & 0 & 0 & ... & 0   & 0 & 0   & 0  & ... \\[1ex]
 \end{tabular}
}}\\[2ex]
\end{tabular}
\end{center}

\begin{corollary}\label{prop1312:2-regular}
Suppose $X\subset \P^{n+e}$ is a non-degenerate variety of dimension $n$ and codimension $e$ with property $\N_{3,e}$. Then,
 $\beta^R_{e,2}=0$ if and only if $X$ is 2-regular.
\end{corollary}
\begin{proof}
Let $L^e$ be a linear space of dimension $e$ and assume that $X\cap L^{e}$ is finite. By Theorem~\ref{3-regulr-multisecants}~(a),
$
\deg(X\cap L^e)\leq 1+e+\beta^R_{e,2}.
$
Therefore, $\beta^R_{e,2}=0$ implies $\deg(X\cap L^e)\leq 1+e$. Since $X$ is a nondegenerate variety this implies that $X$ is small.
(i.e. for every zero-dimensional intersection of $X$ with a linear space $L$, the degree of $X\cap L$ is at most $1+\dim(L)$. (See \cite[Definition~11]{E}.) Then it follows directly from \cite[Theorem~0.4]{EGHP1} that $X$ is $2$-regular. 
\end{proof}

\begin{remark}
What can we say about the case $\beta^R_{\alpha,2}=0$ where $\alpha<e$? In this case, we see that if $\Lambda\cap X$ is finite for a linear subspace $\Lambda$ of dimension $\leq \alpha$ then $\deg(\Lambda\cap X)\leq \dim\Lambda+1$. Note that this condition is a necessary condition for property $\N_{2,\alpha}$. However, the converse is false in general, as for example in the case of a double structure on a line in $\P^3$ or the case of the plane with embedded point. (See  \cite[Example~1.4]{EGHP2}.) We do not know if there are other cases when $X$ is a variety.
\end{remark}

\begin{example}[Macaulay2 \cite{GS}]
(a) The two skew lines $X$ in $\P^3$ satisfies $\deg(X)=2<1+e=3$. The Betti table of $R/I_X$ is given by
\[{\tiny\texttt{
 \begin{tabular}{c|ccccccccccccccccccccccccccccc}
             & 0 & 1 & 2 & 3  & 4 & ... \\[1ex]
        \hline\\
        0 & 1 & 0 & 0  & 0 & 0&... \\[1ex]
        1 & 0 & 4 & 4  & 1 & 0&... \\[1ex]
        2 & 0 & 0 & 0  & 0 & 0&... \\[1ex]
         \end{tabular}
}}\]
Note that $X$ is $2$-regular but not aCM.\\
(b) Let $C$ be a rational normal curve in $\P^4$, which is $2$-regular. If $X=C\cup P$ for a general point $P\in \P^4$ then $\deg(X)=1+e=4$. However a general hyperplane $L$ passing through $P$ is $5$-secant $3$-plane such that $\deg(L\cap X)=5>4=1+e$. This implies that $\beta^R_{e,2}(R/I_X)\neq 0$.
If $P\in \Sec(C)$ then there is a $3$-secant line to $X$. Therefore $\beta^R_{1,2}(R/I_X)\neq 0$. For the two cases, the corresponding Betti tables for $X$ are computed as follows (\cite[Macaulay 2]{GS}):
\begin{center}
\begin{tabular}{ccccccc}

{\tiny\texttt{
 \begin{tabular}{c|ccccccccccccccccccccccccccccc}
             &0 & 1 & 2 & 3 & 4 & 5 & ... \\[1ex]
        \hline\\
        0 & 1 & 0 & 0 & 0   & 0 & 0   & ... \\[1ex]
        1 & 0 & 5 & 5 & 0   & 0  & 0   &... \\[1ex]
        2 & 0 & 1 & 3 & 4   & 1 & 0   & ... \\[1ex]
        3 & 0 & 0 & 0& 0   & 0 & 0   & ... \\[1ex]
  \end{tabular}
}}

& \quad
& {\tiny\texttt{
 \begin{tabular}{c|ccccccccccccccccccccccccccccc}
             &0 & 1 & 2 & 3 & 4 & 5 & ... \\[1ex]
        \hline\\
        0 & 1 & 0 & 0 & 0   & 0 & 0   & ... \\[1ex]
        1 & 0 & 5 & 4 & 0   & 0  & 0   &... \\[1ex]
        2 & 0 & 0 & 3 & 4   & 1 & 0   & ... \\[1ex]
        3 & 0 & 0 & 0& 0   & 0 & 0   & ... \\[1ex]
  \end{tabular}
}} \\
&&&
\\

Case 1: $P\in \Sec(C)$ &  & Case 2: $P\notin \Sec(C)$
\end{tabular}
\end{center}
\end{example}

\begin{example}[F.-O. Schreyer]\label{example:2014} 
Let $C$ be a rational normal curve and $Z$ be a set of general $4$ points  in $\P^3$. \\

{\small\texttt{
\begin{tabular}{lrllllll}
 i1 : &R&=&QQ[x\_0..x\_3];  \\
 &C&=&minors(2,matrix\{\{x\_0,x\_1,x\_2\},\{x\_1,x\_2,x\_3\}\}); -- a rational normal curve\\
 &Z&=&minors(2,random(R\h2,R\h\{4:-1\})); -- general 4 points\\
 &X&=&intersect(C,Z);\\
\end{tabular}
}}\\

Using Macaulay 2, we can compute the Betti table of $X=C\cup Z$ as follows:\\

{\small\texttt{
\begin{tabular}{llllllll}
 i5 & : & betti res X  \\[1ex] 
\end{tabular}
}}\\

{\small\texttt{
\begin{tabular}{llrlllll}
      &     &  &       & 0 & 1 & 2 & 3 \\[1ex]     
o5 &=& total &:& 1 & 6 & 6 & 1 \\[1ex]
      &&            0 &:& 1 & . & .  & . \\[1ex]
      &&             1 &:& . & . & .  & . \\[1ex]
       &&            2 &:& . & 6 & 6  & .\\[1ex]
        &&           3 &:& . & . & .  & 1\\[1ex] 
\end{tabular}
}}\\

Since the codimension $e$ of $X$ is two, $X$ satisfies property $\N_{3,e}$. Note that $X$ is not $3$-regular. Unlike the case of $\N_{2,e}$, the condition $\N_{3,e}$ in general does not imply $3$-regularity.
\end{example}

\subsection{The proof of Theorem~\ref{3-regulr-multisecants} (b)} 
Suppose that $X$ satisfies property $\N_{3,e}$ and let $L^e\subset\P^n$ be a linear space of dimension $e$. If $X\cap L^{e}$ is finite then we have the following inequality from Theorem~\ref{3-regulr-multisecants}~(a) as follows;
\begin{equation}\label{N3e_inequality}
\deg(X\cap L^e)\leq 1+e+\frac{(e-p)(e+p+1)}{2}\leq 1+e+\binom{e+1}{2}=\binom{e+2}{2}.
\end{equation}
This implies that $\deg(X)\leq \binom{2+e}{2}$ since $\deg(X)$ is defined by $\deg(X\cap L^e)$ for a {\it general} linear space $L^e$ of dimension $e$. 

 The bound in \eqref{N3e_inequality} is sharp because if $M$ is a $1$-generic matrix of size $3\times t$  for $t\geq 3$ then the determinantal variety $X$ defined by maximal minors of $M$  acheives this degree bound. In this case, the minimal free resolution of $I_X$ is a $3$-linear resolution, which is given by Eagon-Northcott complex. 

 Note that if $X$ is arithemetically Cohen-Macaulay and $I_X$ has $3$-linear resolution then it was shown that $\deg(X)=\binom{e+2}{2}$ in  \cite[Corollary~1.1]{EG}. The converse is not true in general. For example, let $Y$ be the secant variety of a rational normal curve in $\mathbb P^n$ and let $P$ be a general point in $\mathbb P^n$. Then the algebraic set $X=Y\cup P$ has the geometric degree $\binom{e+2}{2}$ but it does not satisfy $N_{3,e}$ because there exists a $(\binom{e+2}{2}+1)$-secant $e$-plane to $X$. This also implies that $I_X$ does not have $3$-linear resolution.
 
 It is natural to ask what makes the ideal $I_X$  have $3$-linear resolution under the condition $\deg(X)=\binom{e+2}{2}$.  Theorem~\ref{3-regulr-multisecants} (b) shows that property $\N_{3,e}$ is sufficient for this.
  
\begin{remark} Note that the condition $\N_{3,e}$ is essential and cannot be weakened. For example, let $S$ be a smooth complete intersection surface of type $(2,3)$ in $\P^4$. Then the codimension $e$ is two such that $\deg(S)=6=\binom{e+2}{2}$. However $I_X$ does not have $3$-linear resolution. Note that $S$ satisfies $\N_{3,e-1}$ but not $\N_{3,e}$.
\end{remark}

For proof we need the following lemma.
\begin{lemma}\label{cor:303}
Suppose that $X$ satisfies \textit{property} ${\rm {\bf N}}_{3,e}$ and $\deg(X)= \binom{e+2}{2}$.  Then,
\begin{itemize}
 \item[(a)] $I_X$ has no quadric generators. This implies that $I_X$ is $3$-linear up to $e$-th step.
 \item[(b)] $\binom{\alpha+1}{2}\le \beta^R_{\alpha,2}(R/I_X)$ for all $1\le {\alpha}\le e$.
\end{itemize}
\end{lemma}
\begin{proof}
Assume that $\deg(X)=\binom{e+2}{2}$ and there is a quadric hypersurface $Q$ containing $X$. Choose a general linear
subspace $L^e$ of dimension $e$ such that $L^e\varsubsetneq Q$. Then, we may assume that the point $q=(1,0,\cdots,0)$ is contained in $L^e\setminus Q$, and thus
we have a surjective morphism $S_1\oplus S_1(-1)\twoheadrightarrow R/I_X$ (see the proof in \cite[Theorem~4.2]{AKS}). Hence,
the multiplicative map $$\Tor_{0}^{S_{1}}(R/I_X,k)_{1}\st{\times x_{0}}\Tor_{0}^{S_{1}}(R/I_X,k)_{2}$$
is a zero map and $\Tor_{0}^{S_{1}}(R/I_X,k)_{2}=0$. Therefore, as in the proof of Theorem~\ref{3-regulr-multisecants} (a)
(see the sequence (\ref{eq:long exact sequence of tor}) for $\alpha=e$ and $p=0$),
$$
\beta^{S_{\e}}_{0,2}\le \beta^{S_{e}}_{0,1}+\beta^{S_{e-1}}_{0,1}+\cdots\beta^{S_{2}}_{0,1}+\beta^{S_{1}}_{0,2}
=e +(e-1)+\cdots + 2+0= \binom{e+1}{2}-1.
$$
Thus, $\deg(X\cap L^e)\le 1+e+\beta^{S_{\e}}_{0,2}\leq \binom{e+2}{2}-1$ which contradicts our assumption.
So, there is no quadric vanishing on $X$ and the minimal free resolution of $I_X$ is $3$-linear up to $e$-th step. In addition, in the case of $3$-linearity up to $e$-th step, there is no syzygies in degree $2$ and
$$
\beta^{S_{\alpha}}_{0,2}=\beta^{S_{\alpha}}_{0,1}+\beta^{S_{\alpha-1}}_{0,1}+\cdots\beta^{S_{2}}_{0,1}+\beta^{S_{1}}_{0,1}
=\binom{\alpha+1}{2}\le \beta^R_{\alpha,2}(R/I_X),
$$
as we wished.
\end{proof}
To prove Theorem~\ref{3-regulr-multisecants} (b), it suffices to show that $\deg(X)=\binom{e+2}{2}$ implies $I_X$ has a $3$-linear resolution under the condition $\N_{3,e}$ (\cite[Corollary~1.11]{EG}). Our proof is divided into four steps.
\begin{proof}[The proof of Theorem~\ref{3-regulr-multisecants} (b):]

Step {\it I}. First we show that if $H$ is a general linear space of dimension $i$ where $e\leq i\leq n$, then $I_{{X\cap H},{H}}$ cannot have quadric generators.

For general linear space $\Lambda$ of dimension $e$, we see from Remark~\ref{remk:2} that $I_{{X\cap \Lambda},{\Lambda}}$ is $3$-regular. Since $X\cap \Lambda$ is a zero dimensional scheme of
$$\deg(X\cap \Lambda)=\deg(X)=\binom{e+2}{2}=\binom{\codim(X\cap \Lambda,\Lambda)+2}{2},$$
it follows from Lemma~\ref{cor:303} that $I_{{X\cap \Lambda},{\Lambda}}$ has a 3-linear resolution and hence there is no quadric generator in the ideal $I_{{X\cap \Lambda},{\Lambda}}$. This implies that if $H$ is a general linear space of dimension $i$ for some $e\leq i\leq n$, then $I_{{X\cap H},{H}}$ cannot have quadric generators. In particular, if $H=\P^n$ then $I_X$ does not have quadric generators and hence
$$\beta_{k,1}(R/I_X)=0 \text{ for all } k\geq 0.$$
\[
{\begin{tabular}{ccccccc}
{\tiny\texttt{
 \begin{tabular}{c|ccccccccccccccccccccccccccccc}
        & 0 & 1 & ... & e-1 & e & e+1 & e+2 & ... \\[1ex]
        \hline\\
        0 & 1 & 0 & ... & 0   & 0 & 0   & 0&... \\[1ex]
        1 & 0 & * & ... & *   & * & *   & *&... \\[1ex]
        2 & 0 & * & ... & *   & * & *   & *&... \\[1ex]
        3 & 0 & 0 & ... & 0   & 0 & *   & *&... \\[1ex]
 \end{tabular}
}}
&$\Longrightarrow$ &
{\tiny\texttt{
 \begin{tabular}{c|ccccccccccccccccccccccccccccc}
        & 0 & 1 & ... & e-1 & e & e+1 & e+2 &... \\[1ex]
        \hline\\
        0 & 1 & 0 & ... & 0   & 0 & 0   &0& ... \\[1ex]
        1 & 0 & 0 & ... & 0   & 0 & 0   &0& ... \\[1ex]
        2 & 0 & * & ... & *   & * & *   &*& ... \\[1ex]
        3 & 0 & 0 & ... & 0   & 0 & *   &*& ... \\[1ex]
 \end{tabular}
}}\\
&&&
\end{tabular}
}
\]
Step {\it II}. The goal in this step is to show that
$$\beta_{k,3}(I_X)=\beta_{k+1,2}(R/I_X)=0 \text{ for all } k\geq e.$$
\[
\begin{tabular}{ccccccc}
{\tiny\texttt{
 \begin{tabular}{c|ccccccccccccccccccccccccccccc}
       & 0 & 1 & ... & e-1 & e & e+1 & e+2 &... \\[1ex]
        \hline\\
        0 & 1 & 0 & ... & 0   & 0 & 0   &0& ... \\[1ex]
        1 & 0 & 0 & ... & 0   & 0 & 0   &0& ... \\[1ex]
        2 & 0 & * & ... & *   & * & *   &*& ... \\[1ex]
        3 & 0 & 0 & ... & 0   & 0 & *   &*& ... \\[1ex]
 \end{tabular}
}}
$\Longrightarrow$ &
{\tiny\texttt{
 \begin{tabular}{c|ccccccccccccccccccccccccccccc}
        & 0 & 1 & ... & e-1 & e & e+1 & e+2 &... \\[1ex]
        \hline\\
        0 & 1 & 0 & ... & 0   & 0 & 0   &0& ... \\[1ex]
        1 & 0 & 0 & ... & 0   & 0 & 0   &0& ... \\[1ex]
        2 & 0 & * & ... & *   & * & 0   &0& ... \\[1ex]
        3 & 0 & 0 & ... & 0   & 0 & *   &*& ... \\[1ex]
 \end{tabular}
}}\\
\end{tabular}
\]
To show this, we prove that if $k\geq e$ then $\beta_{k,3}(\gin I_X)=0$, where $\gin (I_X)$ is a generic initial ideal of $I_X$ with respect to the reverse lexicographic monomial order. Note that $\beta_{k,3}(\gin (I_X))=0$ implies that $\beta_{k,3}(I_X)=0$ (\cite[Proposition~2.28]{G}).
Let $\mathcal G(\gin(I_X))_d$ be the set of {\it monomial generators} of $\gin(I_X)$ in degree $d$. For each monomial $T$ in $R=k[x_0,\ldots, x_n]$, we denote by $m(T)$ 
$$\max\{ i \geq 0 \mid \text{ a variable } x_i \text{ divides } T\}.$$
Now suppose that
\begin{equation}\label{eq: betti and gin}
\beta_{k,3}(\gin(I_X))\neq 0 \text{ for some } k \geq e,
\end{equation}
and let $k$ be the largest integer satisfying the condition \eqref{eq: betti and gin}. By the result of Eliahou-Kervaire \cite{EK} we see that
$$\beta_{k,3}(\gin(I_X))=\big|\{\,T\in \mathcal G(\gin(I_X))_3 \mid m(T)=k \}\big|.$$
Since $\beta_{k,3}(\gin(I_X))\neq 0$, we can choose a monomial $T\in \mathcal G(\gin(I_X))_3$  such that $m(T)=k$. This implies that $T$ is divided by $x_k$. If $H$ is a general linear space of dimension $k$ then it follows from \cite[Theorem~2.30]{G} that
the ideal
\begin{equation}\label{eq:1312}
\gin(I_{{X\cap H},{H}})=\left[\frac{(\gin(I_X),x_{k+1},\ldots, x_n)}{(x_{k+1},\ldots, x_n)}\right]^{\rm sat}=\left[\frac{(\gin(I_X),x_{k+1},\ldots, x_n)}{(x_{k+1},\ldots, x_n)}\right]_{x_k\to 1}
\end{equation}
has to contain the quadratic monomial $T/x_k$. This means that $X\cap H$ is cut out by a quadric hypersurface, which contradicts the result in Step {\it I}. Hence we conclude that $\beta_{k,3}(I_X)=0$ for all $k\geq e$.\\

Step {\it III}.
We claim that
\begin{equation}
\mathcal G(\gin(I_X))_3=\gin(I_X)_3=k[x_0,\ldots,x_{e-1}]_3.
\end{equation}

By Lemma~\ref{cor:303} and \cite[Proposition~2.28]{G}, we see that
\begin{equation}\label{gin and reduction}
\binom{e+1}{2}\leq \beta_{e,2}(R/I_X) = \beta_{e-1,3}(I_X)\leq \beta_{e-1,3}(\gin(I_X)).
\end{equation}
Since $\beta_{k,3}(\gin(I_X))=0$ for each $k\geq e$, any monomial generator $T\in \mathcal G(\gin(I_X))_3$ cannot be divided by $x_k$ for any $k\geq e$.  Thanks to the result of Eliahou-Kervaire \cite{EK} again,
$$
\begin{array}{llllllllllllllll}
\beta_{e-1,3}(\gin(I_X)) & = &\big|\{ T\in \mathcal G(\gin(I_X))_3 \mid m(T)=e-1 \}\big| \\[1ex]
                         & \leq & \dim_k \big( x_{e-1}\cdot k[x_0,\ldots, x_{e-1}]_2\big)\\[1ex]
                         & = & \ds \binom{e+1}{2}.
\end{array}
$$
By the dimension counting and equation~\eqref{gin and reduction}, we have
$\beta_{e-1,3}(\gin(I_X))= \binom{e+1}{2}$
and thus
\[\{ T\in \mathcal G(\gin(I_X))_3 \mid m(T)=e-1 \}= x_{e-1}\cdot k[x_0,\ldots, x_{e-1}]_2,\]
which implies that $x_{e-1}^3\in \gin(I_X)$. Note that $\gin(I_X)$ does not have any quadratic monomial. Hence we conclude from Borel fixed property of $\gin(I_X)$ that
\begin{equation}
\mathcal G(\gin(I_X))_3=\gin(I_X)_3=k[x_0,\ldots,x_{e-1}]_3.
\end{equation}

\noindent Step {\it IV}. Finally, by the result in Step {\it II}, we only need to show that, for all $k\geq e$ and $j\geq 3$,
$$\beta_{k,j}(I_X)=0.$$
\[
\begin{tabular}{ccccccc}
{\tiny\texttt{
 \begin{tabular}{c|ccccccccccccccccccccccccccccc}
       & 0 & 1 & ... & e-1 & e & e+1 & e+2 &... \\[1ex]
        \hline\\
        0 & 1 & 0 & ... & 0   & 0 & 0   &0& ... \\[1ex]
        1 & 0 & 0 & ... & 0   & 0 & 0   &0& ... \\[1ex]
        2 & 0 & * & ... & *   & * & 0   &0& ... \\[1ex]
        3 & 0 & 0 & ... & 0   & 0 & *   &*& ... \\[1ex]
        4 & 0 & 0 & ... & 0   & 0 & *   &*& ... \\[1ex]
 \end{tabular}
}}
$\Longrightarrow$ &
{\tiny\texttt{
 \begin{tabular}{c|ccccccccccccccccccccccccccccc}
        & 0 & 1 & ... & e-1 & e & e+1 & e+2 &... \\[1ex]
        \hline\\
        0 & 1 & 0 & ... & 0   & 0 & 0   &0& ... \\[1ex]
        1 & 0 & 0 & ... & 0   & 0 & 0   &0& ... \\[1ex]
        2 & 0 & * & ... & *   & * & 0   &0& ... \\[1ex]
        3 & 0 & 0 & ... & 0   & 0 & 0   &0& ... \\[1ex]
        4 & 0 & 0 & ... & 0   & 0 & 0   &0& ... \\[1ex]
 \end{tabular}
}}\\
\end{tabular}
\]
Since $\beta_{k,j}(I_X)\leq \beta_{k,j}(\gin(I_X))$, it is suffices to prove that $\gin(I_X)$ has no generators in degree $\geq 4$. To prove this, suppose that there is a monomial generator $T\in \mathcal G(\gin(I_X))_j$ for some $j\geq 4$. Then the monomial $T$ can be written as a product of two monomial $N_1$ and $N_2$ such that
$$N_1\in k[x_{e},\ldots, x_n], \quad N_2\in k[x_0, \ldots, x_{e-1}].$$
By the result in Step {\it III}, if the monomial $N_2$ is divided by some cubic monomial in $k[x_0, \ldots, x_{e-1}]$ then $T$ cannot be a monomial generator of $\gin(I_X)$. Hence we see $\deg(N_2)$ is at most $2$. If $\Lambda$ is a general linear space  of dimension $e$ then it follows from the similar argument given in the proof of Step {\it III} with the equation \eqref{eq:1312} that $N_2\in \gin(I_{{X\cap \Lambda},{\Lambda}})$. Hence $I_{{X\cap \Lambda},{\Lambda}}$ has a hyperplane or a quadratic polynomial, which contradicts the result proved in Step {\it I}. 
\end{proof}

\begin{remark}
The similar argument in the proof can also be applied to show Theorem~\ref{Eisenbud et al} (b).
\end{remark}

\begin{example}
In \cite{HK}, the authors have shown that if a non-degenerate reduced scheme $X\subset \P^n$ satisfies $\N_{2,p}$ for some $p\geq 1$ then the inner projection from any smooth point of $X$ satisfies at least property $\N_{2,p-1}$. So it is natural to ask whether the inner projection from any smooth point of $X$ satisfies at least property $\N_{3,p-1}$ when $X$ satisfies $\N_{3,p}$ for some $p\geq 1$. Our result shows that this is not true in general. For examples, if we consider the secant variety $X={\rm Sec}(C)$ of a rational normal curve $C$ then the inner projection $Y$ from any smooth point of $X$ has the degree
$$\deg(Y)=\binom{2+e}{2}-1=\binom{e+1}{2}+\binom{e}{1}>\binom{2+(e-1)}{2},$$
where $e={\rm codim(X)}$ and $e-1={\rm codim(Y)}$. This implies that $X$ satisfies $\N_{3,e}$ but $Y$ does not satisfy $\N_{3,e-1}$.
\end{example}

\begin{example}
Remark that there exists an {\it algebraic set} $X$ of degree $<\binom{e+2}{2}$ whose defining ideal $I_X$ has $3$-linear resolution.
For example, let $I=(x_0^3, x_0^2x_1, x_0x_1^2, x_1^3, x_0^2x_2)$ be a monomial ideal of $R=k[x_0,x_1,x_2,x_3]$.
Note that the sufficiently generic distraction $D_{\mathcal L}(I)$ of $I$ is of the form
$$D_{\mathcal L}(I)=(L_1L_2L_3, L_1L_2L_4, L_1L_4L_5, L_4L_5L_6, L_1L_2L_7),$$
where $L_i$ is a generic linear form for each $i=1,\ldots, 7$ (see \cite{BCR} for the definition of distraction). Then the algebraic set $X$ defined by the ideal $D_{\mathcal L}(I)$ is an union of $5$ lines and one point such that its minimal free resolutions are given by

{\tiny
\begin{longtable}{ c  c  c  c l}
\hline\\
 & {\bf $R$-modules} & {\bf $S_1$-modules} & {\bf $S_2$-modules} \\
\hline\\
 & {\tiny\texttt{
 \begin{tabular}{c|ccccccccccccccccccccccccccccc}
           &0 & 1 & 2 & 3 \\[1ex]
        \hline\\
        0 & 1 & 0 &  0 & 0  \\[1ex]
        1 & 0 & 0 &  0 & 0  \\[1ex]
        2 & 0 & 5 &  5 & 1  \\[1ex]
  \end{tabular}
}}  & {\tiny\texttt{
 \begin{tabular}{l|lllll}
           &0 & 1 & 2 &\\[1ex]
        \hline\\
        0 & 1 & 0 & 0\\[1ex]
        1 & 1 & 0 & 0 \\[1ex]
        2 & 1 & 4 & 1 \\[1ex]
  \end{tabular}
}} & {\tiny\texttt{
 \begin{tabular}{l|lllll}
           &0 & 1 \\[1ex]
        \hline\\
        0 & 1 & 0 \\[1ex]
        1 & 2 & 0  \\[1ex]
        2 & 3 & 1  \\[1ex]
  \end{tabular}
}}\\[1ex]
\hline
\end{longtable}
}

   In this case, we see that $e=2$, $\deg(X)=5<\binom{2+2}{2}=6$ and there is a $6$-secant $2$-plane to $X$. We see that a general hyperplane section of $X$ is contained in a quadric hypersurface from $\beta_{e+1,2}(R/I_X)\neq 0$.
\end{example}

From Remark~\ref{remk:2}, we know that if $X$ satisfies $\N_{d, e}$, $(d\geq 2)$ then every linear section $X\cap L^{e}$ of dimension zero is $d$-regular, where $L^{e}$ is a linear space of dimension $e$. Moreover we can verify that
$$\deg(X\cap L^{e})\leq 1+\alpha+ \ds\sum_{t=2}^{d-1} \beta^{S_{e}}_{0,t}\leq  \binom{e+d-1}{d-1}.$$

We close the paper with the following question.\\
{\bf Question.}
 Let $X$ be a non-degenerate algebraic set of dimension $n$ in $\mathbb P^{n+e}$ satisfying $\N_{d,e}$. Then we have $\deg(X)\leq  \binom{e+d-1}{d-1}$. Suppose that equality holds. Is always $X$ a reduced aCM scheme whose defining ideal has $d$-linear resolution?

\begin{center}
{\sc Acknowledgements}
\end{center}
We are thankful to F.-O. Schreyer for personal communications concerning examples of non $3$-regular algebraic sets satisfying $\N_{3,e}$ by using Boij-S\"{o}derberg theory. We are also grateful to the anonymous referees for valuable and helpful suggestions. In addition, the program Macaulay 2 has been useful to us in computations of concrete examples.
\bibliographystyle{plain}

\end{document}